\documentclass[12pt]{amsart}
\usepackage{xcolor}
\usepackage[draft]{todonotes} 
\usepackage[english]{babel}
\usepackage{geometry}
\usepackage[T1]{fontenc}
\usepackage[latin1]{inputenc}
\usepackage{amsmath,amssymb,mathrsfs,amsthm, tikz-cd,mathrsfs}
\usepackage{soul}
\usepackage{epsfig}
\usepackage{hyperref}

\usepackage{graphicx}
\usepackage{xypic}
\usepackage{enumitem}
\usepackage{datetime}
\usepackage{xcolor}
\usepackage{fancyhdr}

\newtheorem{theorem}{Theorem}[section]
\newtheorem*{theorem*}{Theorem}
\newtheorem{proposition}[theorem]{Proposition}

\newtheorem{lemma}[theorem]{Lemma}

\newtheorem{corollary}[theorem]{Corollary}
\newtheorem{definition}[theorem]{Definition}
\newtheorem{example}[theorem]{Example}
\newtheorem{remark}[theorem]{Remark}

\newcommand{\vc}{\|\cdot\|}
\newcommand{\C}{\mathbb{C}}

\newcommand{\Z}{\mathbb{Z}}

\newcommand{\R}{\mathbb{R}}

\newcommand{\p}{\mathbb{P}}
\newcommand{\eps}{\varepsilon}
\newcommand{\X}{\mathcal{X}}

\newcommand{\vf}{\varphi}

\newcommand{\Q}{\mathbb{Q}}
\newcommand{\N}{\mathbb{N}}

\newcommand{\LL}{\mathcal{L}}

\begin{document}
\title{The theta invariants and the  volume function  on arithmetic varieties}

\author{Mounir HAJLI}
\thanks{ }

\address{School of Mathematical Sciences, Shanghai Jiao Tong University, P.R. China}

\email{hajli@sjtu.edu.cn}

            \date{\today, \currenttime}
            
            \maketitle 
            
\begin{abstract}

 We introduce a new arithmetic invariant for hermitian line bundles on an arithmetic variety.   We use this invariant  to measure  the variation of the volume function with respect to the metric.   
 The main result of this paper is a generalized Hodge index theorem on arithmetic toric varieties.

\end{abstract}

{\small{
\begin{center}
MSC: 14G40  \\

Keywords: arithmetic variety; volume function;   theta invariants; arithmetic degree;
\end{center}
}}
 
\tableofcontents

\section{Introduction}

Let $\X$ be    an arithmetic variety over $\mathrm{Spec}(\Z)$, that is  a projective, integral and flat scheme  over $\Z$. We assume  that $\X_\Q$ is smooth over $\Q$.  Let $n+1$ be the absolute dimension of $\X$.     Let $\overline{\mathcal{L}}=(\LL,\|\cdot\|)$ be a  continuous hermitian line bundle on $\X$.
For any $k\in \N_{\geq 1}$, 
$k\overline{\mathcal{L}}$ denotes $\overline{\mathcal{L}}^{\otimes k}$.\\

Moriwaki in \cite{Moriwaki1} introduced the arithmetic volume  $\widehat{\mathrm{vol}}(\overline{\mathcal{L}})$ for  $\mathcal{C}^\infty$ hermitian line bundle  $\overline{\mathcal{L}}$    which is   an analogue of the geometric volume function. The  arithmetic volume  $\widehat{\mathrm{vol}}(\overline{\mathcal{L}})$  is defined by 
\[
\widehat{\mathrm{vol}}(\overline{\mathcal{L}})=
\underset{k\rightarrow \infty}{\limsup} \frac{\log \# \{ s\in H^0(\X,k\mathcal{L}) \mid \|s\|_{\sup,\phi}\leq 1 \} 
}{k^{n+1}/(n+1)!}.
\]
Chen  in \cite{Chen1} proved  that the arithmetic  volume function is actually a limit.  \\

Gillet and Soul\'e in \cite{AIT} defined arithmetic Chow groups $\widehat{CH}^p(\X) $ for any $p\geq 0$. Let 
$(E,h)$ be a hermitian vector bundle on $\X$. When $h$ is smooth, one can   attach    to $(E, h)$,  arithmetic characteristic classes such as  the $(n+1)$-th power of  first Chern character $\widehat{c}_1(E,h)^{n+1}\in \widehat{CH}^{n+1}(\X)_\Q$, see \cite{BoGS,Character}.\\

If $\X$ is regular, $\mathcal{L}$ is ample on $\X$ 
and the metric of  $\overline{\mathcal{L}}$ is  smooth  with positive  first Chern form $c_1(\overline{\mathcal{L}})$  on $X$, then
\begin{equation}\label{ineq}
\widehat{\mathrm{vol}}(\overline{\mathcal{L}})\geq \widehat{\deg}(\widehat c_1(\overline{\mathcal{L}})^{n+1}).
\end{equation}
This inequality can be obtained using  the arithmetic Riemann-Roch theorem due to Gillet-Soul\'e \cite{ARR},  or the arithmetic Hilbert-Samuel formula due to Abbes-Bouche \cite{Abbes}. In \cite{Moriwaki2,MoriwakiIMRN},  Moriwaki proved that the arithmetic volume function is continuous with respect to
$\overline \LL
$. \\

 The 
  $\chi$-arithmetic volume of  
$\widehat{\mathrm{vol}}_\chi(\overline{\LL})$ is defined as follows

\[
\widehat{\mathrm{vol}}_\chi(\overline{\LL})= \lim_{k\rightarrow \infty} \frac{\hat{\chi}(H^0(\X,k\LL),\|\cdot\|_{\sup, k\phi})
}{k^{n+1}/(n+1)!},\]
(for the definition of  $\hat{\chi}(H^0(\X,k\LL),\|\cdot\|_{\sup, k\phi})$, see for instance \cite{MoriwakiAdelic,MoriwakiAMS}).\\

It is known that the inequality 
\[
\widehat{\mathrm{vol}}(\overline{\LL})\geq \widehat{\mathrm{vol}}_\chi(\overline{\LL}).
\]
holds for any hermitian line bundle $\overline\LL$ on $\X$.\\

When $\X=\p^1_\Z$, and $\overline \LL$ is a toric DSP line bundle on $\p^1_\Z$, that is a difference of semipositive ones such that metric is invariant under the action of the compact torus of $\p^1$ (see \cite{Burgos3} for more details), then 
\[
\widehat{\mathrm{vol}}_\chi(\overline{\LL}) 
\geq \widehat{\deg}(\hat c_1(\overline \LL)^2),
\]
(see \cite{MounirIJNT}). \\

According to \cite{Moriwaki1}, there are three kinds of positivity of $\overline{\mathcal{L}}=(\LL,\|\cdot\|)$.
\begin{itemize}
\item \textit{ample} : $\overline{\mathcal{L}}$ is ample if $\LL$ is ample on $\X$, 
the first Chern form $c_1(\overline{\LL})$ 
is positive on $\X(\C)$ and, for a sufficiently large integer $k$, 
$H^0(\X, k \LL)$ is generated by the 
set 
\[
\{s\in H^0(\X,k\LL)\mid \|s\|_{\sup}<1 \},
\]
as a $\Z$-module.
\item \textit{nef} : $\overline{\LL}$ 
is nef if the first Chern form $c_1(\overline {\LL})$ is semipositive and 
$\widehat{\deg}(\overline{\LL}_{|_\Gamma})\geq 0$ for any $1$-dimensional closed subscheme $\Gamma$ in $\X$.

\item \textit{big} : $\overline{\LL}$ is big 
if  $\LL_\Q$ is big on $\X_\Q$  and 
there is a positive integer $k$ and a 
non-zero section $s$ of $H^0(\X,k\LL)$ 
with
$\|s\|_{\sup}<1$.\\
  
\end{itemize}

Let $\overline{\LL}=(\LL,\|\cdot\|_\phi)$ and $\overline{\mathcal{N}}=(\mathcal{N},\|\cdot\|_{\psi})$ be two  $\mathcal{C}^\infty$-hermitian line bundles on $\X$.  If $\overline\LL$ is ample,   we have the following asymptotic expansion 
  \begin{equation}\label{asym1}
\hat{\chi}(H^0(\X,k\LL+\mathcal{N}),\|\cdot\|_{\sup, k\phi+\psi})=\frac{1}{(n+1)!} \widehat{\deg}(\widehat{c}_1(\overline{\LL})^{n+1}) k^{n+1}+o(k^{n+1}),
\end{equation}
(see \cite{Abbes,Amplitude,ZhangPositive}). As a consequence,
\begin{equation}\label{asym}
\hat{h}^0(H^0(\X,k\LL+ N ),\|\cdot\|_{\sup, k\phi+\psi} )=\frac{1}{(n+1)!} \widehat{\deg}(\widehat{c}_1(\overline{\LL})^{n+1}) k^{n+1}+o(k^{n+1}),
\end{equation}
as $k\rightarrow \infty$, see  \cite[Lemma 3.3]{Moriwaki2}.  It is worth pointing out that  the proof of \eqref{asym}  relies heavily  on    the asymptotic  \eqref{asym1}.\\

Moriwaki  generalized  
\eqref{ineq} to $\mathcal{C}^\infty$-hermitian line bundles $\overline{\mathcal{L}}$ on $\X$ such that
$\mathcal{L}$ is nef on the generic fiber of $\X$, $c_1(\overline{\mathcal{L}})$ is semipositive on $X$, and $\mathcal{L}$ has 
moderate growth of positive even cohomologies, see \cite[Theorem A]{Moriwaki2}. If $\overline{\mathcal{L}}$ is nef, then 
\[
\widehat{\mathrm{vol}}(\overline{\mathcal{L}})=\widehat{\deg}(\widehat c_1(\overline{\mathcal{L}})^{n+1}),
\]
see \cite[Theorem C]{Moriwaki2}.\\

One of the goals of this paper is to recover the results of Abbes-Bouche, Moriwaki, and Zhang using a different approach.  Let  $\overline{\LL}=(\LL,\|\cdot\|_\infty)$  be a  hermitian line bundle 
on an arithmetic variety $\X$ endowed with a measure $\mu$. To this data,  we attach   the invariant $\Theta(\mu, \phi)$, see
Definition \ref{defTheta}. This new invariant can be seen as an arithmetic  analogue of the distortion function $\rho(\mu,\phi)$. We prove in  Corollary \ref{upperTheta} that
\[
\underset{k\rightarrow \infty}{\limsup}  \frac{\Theta(\mu,k\phi)}{k^n} \mu
\leq \mu_{\mathrm{eq}}(X,\phi).
\]
with equality when  $\overline \LL$ is ample, see Theorem \ref{ample}, where
$\mu_{\mathrm{eq}}(X,\phi)$ is the equilibrium measure of $\phi$. \\

The introduction of $\Theta(\mu,\phi)$
allows us to recover \eqref{ineq} in the case when $\X$ is toric, and $\LL$ is an equivariant line bundle on $\X$. More
precisely, we obtain the result.

\begin{theorem}[Main Theorem](see Theorem \ref{Hodge})
 Let $\X$ be an arithmetic toric variety over  $\Z$ of relative dimension $n$. Let  $\mathcal{L}$ be an equivariant ample line bundle  on $\X$.
We   assume that  the metric of  $\overline{\mathcal{L}}$ is  smooth  with semi-positive  first Chern form $c_1(\overline{\mathcal{L}})$  on $\X$, then
\[
\widehat{\mathrm{vol}}(\overline{\mathcal{L}})\geq \widehat{\deg}(\widehat c_1(\overline{\mathcal{L}})^{n+1}).
\]
\end{theorem}

\section{The distortion function and  the variation of the arithmetic degree}

   Let $X$ be a compact complex manifold of dimension $n$.
Let $\mu$ be  a probability measure with non-pluripolar support on $X$. Let 
 $L$ be a holomorphic line bundle on $X$. A weight $\phi$ on $L$ is a locally integrable function on the complement of the zero-section in the total space of the dual line bundle $L^\ast$ satisfying the log-homogeneity property 
\[
\phi(\lambda v)=\log|\lambda|+\phi(v)
\]
for all non-zero $v\in L^\ast, \lambda\in \C$.  Let $\phi$ be a weight function on $L$. $\phi$ 
defines a hermitian metric  on $L$, which we denote by $\|\cdot\|_\phi$.\\

 Let  $(L,\|\cdot\|_\phi)$ be  a hermitian line bundle on 
$X$, where   $\phi$ is  the weight of the metric of $L$. 
We endow the space of global sections $H^0(X,L) $ with the $L^2$-norm
\[
\|s\|_{(\mu, \phi)}^2:=\int_X \|s(x)\|_{\phi}^2\mu.
\]
Also we consider the sup norm defined as follows
\[
\|s\|_{\sup,  \phi}:=\sup_{x\in X}\|s(x)\|_{\phi}.
\]
for any $s\in H^0(X,L)$. \\

The Bergman distortion function $\rho(\mu,\overline{L})$ is by definition the function given at a point $x\in X$ by
\[
\rho(\mu,\phi)(x)=\sup_{s\in H^0(X,L)\setminus\{0\}}\frac{\|s(x)\|_{\phi}^2}{\quad\|s\|_{(\mu, \phi)}^2}.
\]
If $\{s_1,\ldots,s_N\}$ is a $(\mu, \phi)$-orthonormal basis of $H^0(X,{L})$, then it is
 well known that
 \[
 \rho(\mu,\phi)(x)=\sum_{j=1}^N \|s_j(x)\|_{\phi}^2\quad \forall\,x\in X.
 \]


 We say that  $\mu$ has the Bernstein-Markov property with respect to  $\|\cdot\|_{\phi}$ if  for any $\eps>0$ we have
 \[
 \sup_X \rho(\mu,k\phi)^\frac{1}{2}=O(e^{k\eps}).
 \]

 If $\mu$ is a smooth positive volume form and $\|\cdot\|_{\phi}$ is a continuous metric on $L$ then
 $\mu$ has the Bernstein-Markov property with respect to $\|\cdot\|_{\phi}$  (see \cite[Lemma 3.2]{BermanBoucksom}).\\

Let $\phi$ be a weight
of a continuous hermitian metric $\|\cdot\|_\phi$
on $L$. When $\phi$ is smooth, we define the Monge-Amp\`ere operator as
\[
\mathrm{MA}(\phi):=(dd^c\phi)^{\wedge n}.
\]
The equilibrium weight of $\phi$ is defined as
\[
P_X\phi:=\sup\{\psi\mid\,\psi\, \text{psh weight on}\, L,\, \psi\leq \phi\,\text{on}\, X \}.
\]
It is  known that the equilibrium weight is upper semicontinuous psh weight
  with
minimal singularities. The equilibrium  measure of $\phi$ is defined by
\[
\mu_{\mathrm{eq}}(\phi):=\frac{1}{\mathrm{vol}(L)} \mathrm{MA}(P_X\phi).
\]

Let $\mu$  be a smooth positive volume form on $X$ and 
$\phi$ a $\mathcal{C}^2$ weight on $L$. We have 
\[
\lim_{k\rightarrow \infty} \frac{1}{k^{\dim X}} \rho( \mu,k\phi)\mu=\frac{1}{\mathrm{vol}(L)} \mathrm{MA}(P_X\phi),
\]
in the weak topology of measures (see for instance \cite[Theorem 3.1]{BermanBoucksom}). 
\\ 

Let $\|\cdot\|_{\overline{L}_0}$ and $\|\cdot\|_{\overline{L}_1}$ be two smooth  hermitian  metrics on $L$. We define the Monge-Amp\`ere
functional $\mathcal{E}$ by the formula

\begin{equation}\label{AF}
\mathcal{E}(\overline{L}_1)-\mathcal{E}(\overline{L}_0):=\frac{1}{n+1}\sum_{j=0}^n\int_X
-\log\frac{\|\cdot\|_{\overline{L}_1}}{\|\cdot\|_{\overline{L}_0}}c_1(\overline{L}_0)^{\wedge j}\wedge c_1( \overline{L}_1)^{\wedge n-j}.
\end{equation}

An admissible metric  $\vc$ on a holomorphic line bundle $L$ is, by definition,  a uniform limit of a sequence $\bigl(\vc_n \bigr)_{n\in \N}$ of positive and smooth hermitian metrics  $L$. An admissible line bundle $(L,\|\cdot\|)$ is a line bundle $L$ endowed with an admissible metric   $\|\cdot\|$. We say that $\overline{L}=(L,\|\cdot\|)$ is an integrable line bundle if there exist  $\overline{L}_1$ and
 $\overline{L}_2$, two admissible line bundles  on $X$ such that
\[
 \overline{L}=\overline{L}_1\otimes \overline{L}_2^{-1}.
\]

By the theory of Bedford-Taylor \cite{Bedford},  
\eqref{AF} extends to admissible metrics,  and hence to
integrable ones by polarisation.\\ 

Let  $\phi_{\overline{L}_0}$ and $\phi_{\overline{L}_1}$ be 
 the associated 
weights of $\|\cdot\|_{\overline{L}_0}$ and $\|\cdot\|_{\overline{L}_1}$ respectively.   Following \cite{BermanBoucksom},   when $L$ is big we set
 \[
  \mathcal{E}_{\rm{eq}}(\overline{L}_1)-\mathcal{E}_{\rm{eq}}
(\overline{L}_0):=\frac{1}{\mathrm{Vol}(L)}(\mathcal{E}((\overline{L}_1)_X)-\mathcal{E}((\overline{L}_0)_X).
 \]
 where $(\overline{L}_i)_X$ denotes the line bundle $L$ endowed with the weight $P_X\phi_{\overline{L}_i}$, the equilibrium  weight of $\phi_{\overline{L}_i}$ for $i=0,1$.
In \cite[\S 1.3]{BermanBoucksom}, $\mathcal{E}_{\rm{eq}}(\overline{L})$ is called the energy at equilibrium of $(X,\phi_{\overline{L}})$
($\phi_{\overline{L}}$ is the weight of $\overline{L}$).\\



\begin{theorem}
Let $\X$ be an arithmetic variety over $\Z$ of dimension $n+1$. Let
$\LL$ be a line bundle on $\X$. Let 
$\|\cdot\|$ and $\|\cdot\|'$ be two integrable 
metrics on $\LL$. We have 
\begin{equation}\label{vardegree}
\widehat{\deg}(\widehat c_1(\mathcal{L},\|\cdot\|  )^{n+1})-\widehat{\deg}(\widehat c_1(\mathcal{L}, \|\cdot\|')^{n+1})= \sum_{i+j=p-1}
\int_{\X(\C)} \vf c_1(\LL,\|\cdot\|)^i c_1(\LL,\|\cdot\|')^j .
\end{equation}
where $\vf$ is such that $\|\cdot\|'=e^{\vf}\|\cdot\|$

\end{theorem}

\begin{proof}
Let $\|\cdot\|$ and $\|\cdot\|'$ be two smooth metrics on $\LL$. where $\vf$ is such that $\|\cdot\|'=e^{\vf}\|\cdot\|$, see \cite[Proposition 3.2.2]{BoGS}. When the metrics  are integrable,  then   \eqref{vardegree}  can be generalized to  integrable metrics. This is  an easy combination of    \cite[Proposition 5.5.2, (2),(3)]{Maillot}, and \eqref{vardegree}.

\end{proof}

\section{The theta invariants associated with euclidean lattices } 


Let $\overline V$ be a hermitian vector bundle over $\mathrm{Spec}(\Z)$, that is a finitely generated $\Z$-module $V$ which is equipped with a hermitian norm which is invariant under complex conjugation, 
 on the complex vector space
\[
V\otimes_{\Z}\C.
\]

Let  $\lambda_{\overline V}$ be the  unique translation-invariant Radon measure on $V_\R$ which satisfies the following normalization condition:  for every orthonormal basis $\{e_1,\ldots,e_N\}$ of $(V_\R,\|\cdot\|_{\overline V})$,
\[
\lambda_{\overline V} \left(\sum_{i=1}^N [0,1[e_i \right)=1.
\] 
We set
\[
\mathrm{covol}(\overline V):=\lambda_{\overline V}(\sum_{i=1}^N [0,1[v_i),
\]
for every $\Z$-basis $\{v_1,\ldots,v_N\}$ of $V$. $\mathrm{covol}(\overline V)$ is called the covolume of $\overline V$.  \\

We set
\[
\theta_{\overline V}(t)=\sum_{v\in V} e^{-\pi t  \|v\|_{\overline V}^2}.
\]

We have the following identity
\begin{equation}\label{poisson}
\sum_{v\in V} e^{-\pi \|v\|_{\overline V}^2}=(\mathrm{covol}(\overline V))^{-1} \sum_{v^\vee \in V^\vee} e^{-\pi \|v^\vee\|_{\overline V^\vee}^2},
\end{equation}
which is a consequence of the Poisson formula (\cite[(2.1.2)]{BostTheta}).\\

Let $\overline E=(E,\|\cdot\|)$ be  hermitian vector bundle over $\Z$. $\overline E$ is called also an euclidean lattice. We let
\[
h_\theta^0(\overline E):=\log \sum_{v\in E} e^{-\pi \|v\|_E^2},
\]
and
\[
h_\theta^1(\overline E^\vee):=\log \sum_{v^\vee\in E^\vee} e^{-\pi \|v^\vee\|_{\overline E^\vee}^2},
\]
and 
\[
\widehat{\deg}(\overline E):=-\log \mathrm{covol}(\overline E),
\]
and
\[
h^0_{\mathrm{Ar}}(\overline{E}):=\log \left| \{ v\in E\mid \|v\|\leq 1 \right|.
\]

 The equation \eqref{poisson} may be written in terms of the 
 $
 \theta$-invariants $\hat h_\theta^0(\overline E)$ and $\hat h_\theta^0(\overline E)$, and the Arakelov degree $\widehat{\deg}(\overline E)$ as follows 
 \[
h_\theta^0(\overline E)-h_\theta^1(\overline E^\vee)=\widehat{\deg}(\overline E).
\]

On the other hand, we have 
\begin{equation}\label{ArTheta}
h_\theta^0(\overline E)-\frac{1}{2}\mathrm{rk} \;E \log \mathrm{rk}\;E+\log(1-\frac{1}{2\pi})\leq h_{\mathrm{Ar}}^0(\overline E)\leq h_\theta^0(\overline E)+\pi,
\end{equation}
(see \cite[Theorem 3.1.1]{BostTheta}).

\begin{lemma}\label{lemma1}

For $t>0$,
\[
\left | \log \theta_{\overline{E}}(t)-\log \theta_{\overline E}(1) \right| 
\leq \frac{1}{2}\mathrm{rk}\; E \cdot \log t .
\]
\end{lemma}

\begin{proof}
This lemma follows from the fact that $\log \theta_{\overline{E}}(t) $ 
is a decreasing function of $t$ in $\R_+^\ast$, and 
\[
\log \theta_{\overline{E}}(t) +\frac{1}{2} \mathrm{rk}\;E\cdot \log t
\]
is an increasing function of $t$ in $\R_+^\ast$, see \cite[Lemma 3.1.4]{BostTheta}.
\end{proof}

\begin{lemma}\label{lemma2}
For all $t>0$, we have
\begin{equation}\label{ineq2}
\sum_{v\in E} \|v\|^2 e^{-\pi t \|v\|^2}\leq \frac{\mathrm{rank}(E)}{2\pi t} \sum_{v\in E} e^{-\pi t \|v\|^2}.
\end{equation}
\end{lemma}

\begin{proof}
See  \cite[(3.1.5)]{BostTheta}. 
\end{proof}

\section{Theta invariants of hermitian line bundles on arithmetic varieties }

Let $\X$ be an arithmetic variety over $\Z$ of dimension $n+1$. Let
$\overline{\LL}=(\LL,\|\cdot\|)$ be a hermitian  line bundle on $\X$.
We assume that the $\Z$-module  $H^0(\X,k\LL)$ is  a
torsion-free for every $k\in \N$.  For any $k\in \N$, $N_k$ denotes    the rank of $H^0(\X,k\LL)$.
We set $X:=\X(\C)$, and $L:=\mathcal{L}(\C)$. Let $\mu$ be  a probability measure with non-pluripolar support on $X$.
We denote by $\phi$ the weight of the metric of  $\overline{\mathcal{L}}$, and we write 
$\|\cdot\|_\phi$ instead of $\|\cdot\|$.\\

For any $k\geq 1$, let $\overline{H^0(\X,k\mathcal{L})}_{(\mu, k\phi )} $ (resp. $ \overline{H^0(\X,k\mathcal{L})}_{(\sup, k\phi )}$) be the hermitian vector bundle 
$H^0(\X,k\mathcal{L})$ over $\Z$ equipped 
with the  $L^2$-norm 
$\|\cdot\|_{(\mu,k\phi)}$ (resp. the sup-norm $\|\cdot\|_{\sup,k\phi}$).\\

We let
\[
h_\theta^0\left(
\overline{H^0(\X,k\mathcal{L})}_{(\mu, k\phi )} \right):=
\log \sum_{v\in H^0(\X,k\mathcal{L})} e^{-\pi \|v\|_{(\mu,k\phi)}^2},
\]
and
\[
h_\theta^0\left(
\overline{H^0(\X,k\mathcal{L})}_{\sup, k\phi}\right):=
\log \sum_{v\in H^0(\X,k\mathcal{L})} e^{-\pi \|v\|_{\sup,k\phi}^2}.
\]

\begin{theorem}\label{thm3.1}

Assume that $\mu$ has the Bernstein-Markov property with respect to  $\|\cdot\|_{\phi}$. We have 
\begin{equation}\label{vol1}
\underset{k\rightarrow \infty}{\limsup} \frac{h_\theta^0(
\overline{H^0(\X,k\mathcal{L})}_{(\mu, k\phi )})}{k^{n+1}/(n+1)!} = \underset{k\rightarrow \infty}{\liminf} 
 \frac{h_\theta^0(
\overline{H^0(\X,k\mathcal{L})}_{(\mu, k\phi )})}{k^{n+1}/(n+1)!}=\widehat{\mathrm{vol}}(\overline{\mathcal{L}}), 
\end{equation}
and
\begin{equation}\label{vol2}
\underset{k\rightarrow \infty}{\limsup}\frac{h_\theta^0(
\overline{H^0(\X,k\mathcal{L})}_{\sup, k\phi})}{ k^{n+1}/(n+1)!}
=\underset{k\rightarrow \infty}{\limsup}\frac{h_\theta^0(
\overline{H^0(\X,k\mathcal{L})}_{\sup, k\phi})}{ k^{n+1}/(n+1)!}=\widehat{\mathrm{vol}}(\overline{\mathcal{L}}).
\end{equation}

\end{theorem}

The proof of this theorem is a consequence of the following. 

\begin{proposition}  
Assume that $\mu$ has the Bernstein-Markov property with respect to  $\|\cdot\|_{\phi}$. There exists a positive constant $C\geq 1$ such that for any  $\eps>0$, and any 
$k\gg 1$,

\[
 h_\theta^0(
\overline{H^0(\X,k\mathcal{L})}_{(\mu, k\phi )})-
h_\theta^0(
\overline{H^0(\X,k\mathcal{L})}_{\sup, k\phi})= \eps k O(k^n)+O(k^n).
\]

\end{proposition}

\begin{proof} Let $\eps>0$.  We have,  \[ \|v\|_{\sup,k \phi}  \leq C e^{\eps k}  \|v\|_{(\mu,k\phi)},\quad \forall \;v\in H^0(\X,k\mathcal{L})\quad\forall k\gg 1,\]where $C$ is a positive  constant (independent on $k$).  Then

\[
 h_\theta^0(\overline{H^0(\X,k\mathcal{L})}_{(\mu, k\left(\phi-\eps-\frac{1}{k}\log C \right) }) \leq h_\theta^0(\overline{H^0(\X,k\mathcal{L})}_{\sup, k\phi}) \leq 
 h_\theta^0(\overline{H^0(\X,k\mathcal{L})}_{(\mu, k\phi )})\quad\forall k\gg 1.
\]

Using  Lemma \ref{lemma1}, we get
\[
0\leq h_\theta^0(\overline{H^0(\X,k\mathcal{L})}_{(\mu, k\phi )})-h_\theta^0(\overline{H^0(\X,k\mathcal{L})}_{(\mu, k\left(\phi-\eps-\frac{1}{k}\log C \right))}) \leq N_k (k\eps+C)\quad\forall k\gg 1.
\]

Therefore, for $k\gg 1$,
  \[ 
  \begin{split} 
  h_\theta^0(\overline{H^0(\X,k\mathcal{L})}_{(\mu, k\phi )})  -
  N_k (k\eps+C)\leq   &h_\theta^0(\overline{H^0(\X,k\mathcal{L})}_{\sup, k\phi}) \leq h_\theta^0(\overline{H^0(\X,k\mathcal{L})}_{(\mu, k\phi )}).
  \end{split}\]
  
  Since $N_k=O(k^n)$, we conclude that

\[
 h_\theta^0(
\overline{H^0(\X,k\mathcal{L})}_{(\mu, k\phi )})-
h_\theta^0(
\overline{H^0(\X,k\mathcal{L})}_{\sup, k\phi})= \eps k O(k^n)+O(k^n)\quad\forall k\gg 1.
\]

It follows that
\[
 \underset{k\rightarrow \infty}{\limsup} \frac{h_\theta^0(
\overline{H^0(\X,k\mathcal{L})}_{(\mu, k\phi )})}{k^{n+1}/(n+1)!}-  \underset{k\rightarrow \infty}{\limsup} \frac{h_\theta^0(
\overline{H^0(\X,k\mathcal{L})}_{(\sup, k\phi )})}{k^{n+1}/(n+1)!}  = O(\eps).
  \]
Hence
\[
 \underset{k\rightarrow \infty}{\limsup} \frac{h_\theta^0(
\overline{H^0(\X,k\mathcal{L})}_{(\mu, k\phi )})}{k^{n+1}/(n+1)!}= \underset{k\rightarrow \infty}{\limsup} \frac{h_\theta^0(
\overline{H^0(\X,k\mathcal{L})}_{(\sup, k\phi )})}{k^{n+1}/(n+1)!}.
\]
By \eqref{ArTheta}, we see that
\[
 \underset{k\rightarrow \infty}{\limsup} \frac{h_\theta^0(
\overline{H^0(\X,k\mathcal{L})}_{(\sup, k\phi )})}{k^{n+1}/(n+1)!}= \widehat{\mathrm{vol}}(\overline \LL).
\]
This completes the proof of \eqref{vol1}.
In a similar way, we obtain \eqref{vol2}.

 \end{proof}

\begin{remark}
For every  $s\in H^0(\X,k\LL)$, we have
\[
\|s\|_{\sup,k\phi}=\|s\|_{\sup,k(P_X\phi)},
\]
(this is a special case of  \cite[Proposition 2.8]{BermanBoucksom}). It follows from Theorem \ref{thm3.1}, that
\[
\widehat{\mathrm{vol}}(\overline{\mathcal{L}})=\widehat{\mathrm{vol}}(\overline{\mathcal{L}}_X),
\]
where $\overline{\mathcal{L}}_X=(\LL,\|\cdot\|_{P_X\phi})$.
\end{remark}


\begin{definition}\label{defTheta}

Let $t>0$, and $x\in \X(\C)$. Set
\[
\Theta(\mu,\phi)(t;x):=2\pi\sum_{v\in H^0(\X,\mathcal{L})} \|v(x)\|_\phi^2    \frac{e^{-\pi  t\|v\|_{(\mu,\phi)}^2}}{
\underset{u\in H^0(\X,\mathcal{L})}{\sum} e^{-\pi  t\|u\|^2_{(\mu,\phi) }}}.
\]
When $t=1$, we write $\Theta(\mu,k\phi)(x)$ instead of  
$\Theta(\mu,k\phi)(1;x)$. 

\end{definition}

\begin{proposition} Assume that $\mu$ has the Bernstein-Markov property with respect to  $\|\cdot\|_{\phi}$.  For every $t>0$ and
$k\in \N$
\[
\Theta(\mu,k\phi)(t;x)<\infty \quad \forall x\in X,
\] 
and
\[
\int_X \Theta(\mu,k\phi)(t;x) d\mu\leq \frac{1}{ t}N_k.
\]
Moreover,
\[
\int_X \Theta(\mu,k\phi)(t;x) d\mu=U_{\overline{H^0(\X,k\LL)}_{(\mu,k\phi)}}(t)
\]
see \cite[p. 64]{BostTheta} for the definition
of $U$.
\end{proposition}
\begin{proof}

Let $\eps>0$. For $k\gg 1$ \[ \|v\|_{\sup,k \phi}  \leq C e^{\eps k}  \|v\|_{(\mu,k\phi)},\quad \forall \;v\in H^0(\X,k\mathcal{L}),\]where $C$ is a positive  constant (independent on $k$). 

Then
\[
\Theta(\mu,k\phi)(t;x)\leq  2\pi Ce^{ \eps k}\sum_{v\in H^0(\X,\mathcal{L})} \|v\|_{(\mu,k\phi)}^2    \frac{e^{-\pi  t\|v\|_{(\mu,k\phi)}^2}}{
\underset{u\in H^0(\X,k\mathcal{L})}{\sum} e^{-\pi  t\|u\|^2_{(\mu,k\phi) }}}.
\]
By  Lemma \ref{lemma2}, the right hand side of the preceding  inequality is bounded .\\

It is easy to see that
\[
\int_X \Theta(\mu,k\phi)(t;x) d\mu=2\pi 
\sum_{v\in H^0(\X,k\mathcal{L})} \|v\|_{(\mu,k\phi)}^2    \frac{e^{-\pi  t\|v\|_{(\mu,k\phi)}^2}}{
\underset{u\in H^0(\X,k\mathcal{L})}{\sum} e^{-\pi  t\|u\|^2_{(\mu,k\phi) }}}.
\]
Using again Lemma \ref{lemma2}, we conclude that
\[
\int_X \Theta(\mu,k\phi)(t;x) d\mu\leq \frac{1}{ t}N_k.
\]

\end{proof}

\begin{proposition}\label{difference} Let $\psi$ and $\phi$ be two continuous weights 
on $\LL$. We have 
\[
\hat  h_\theta^0(\overline{H^0(\X,k\mathcal{L})}_{(\mu, k\phi )})- 
\hat  h_\theta^0(\overline{H^0(\X,k\mathcal{L})}_{(\mu, k\psi )})
=k\int_X  (\phi-\psi)\int_0^1 \Theta(x,k \phi_t)d\mu dt\quad\forall k\in \N,
\]
where $\phi_t=t\phi +(1-t)\psi$ with $t\in[0,1]$.

\end{proposition}
\begin{proof} 
Let $\delta$ be a continuous function on $X$. For any 
$s\in \R$, $\phi+s\delta$ defines a weight on $\LL$. \\

It is clear that $h_\theta^0( \overline{H^0(\X,k\mathcal{L})}_{(\mu, k(\phi+s\delta) )})$ as a function of $s\in \R$ is 
differentiable. For any $s\in \R$,
\[
\begin{split}
\frac{d}{ds}  \hat  h_\theta^0( &\overline{H^0(\X,k\mathcal{L})}_{(\mu, k(\phi+s\delta) )})\\
=&\frac{2\pi}{ \sum_{v\in H^0(\X,k\mathcal{L})} e^{-\pi \|v\|^2_{k(\phi +s\delta)}}}\sum_{v\in H^0(\X,k\mathcal{L})} \left(\int_X k \delta(x) \|v(x)\|_{k(\phi+s \delta)}^2 d\mu\right) e^{-\pi \|v\|^2_{k(\phi +s\delta)}}
.\end{split}
\]

At $s=0$, we obtain
\[
\begin{split}
\frac{d}{ds} \hat  h_\theta^0(\overline{H^0(\X,k\mathcal{L})}_{(\mu, k(\phi+s\delta) )})_{|_{s=0}}
 =&2\pi \frac{\sum_{v\in H^0(\X,k\mathcal{L})} \left(\int_X k \delta(x) |v(x)|_{k\phi}^2 d\mu\right) e^{-\pi \|v\|^2_{k\phi }}}{ \sum_{v\in H^0(\X,k\mathcal{L})} e^{-\pi \|v\|^2_{k\phi}}}\\
=& k \int_X \delta(x) \Theta(\mu,k\phi) d\mu.
\end{split}
\]
Hence
\[
\hat  h_\theta^0(\overline{H^0(\X,k\mathcal{L})}_{(\mu, k\phi )})- 
\hat  h_\theta^0(\overline{H^0(\X,k\mathcal{L})}_{(\mu, k\psi )})= 
k \int_X  (\phi-\psi)\int_0^1 \Theta(\mu,k \phi_s)d\mu ds,
\]
where $\phi_s=\psi+s(\phi-\psi)$ for any $s\in \R$.
\end{proof}

\begin{theorem}\label{thetarho}
We have
\[
\Theta(\mu,\phi)(x)\leq   \rho(\mu,\phi)(x) \quad\forall x\in X.
\]
\end{theorem}
\begin{proof}

Let $x\in X$.  Let $s\geq 0$. Let us denote by 
$\overline V_s$  the euclidean lattice $H^0(\X,\mathcal{L})$ endowed with the norm $\|\cdot\|_{\overline{V}_s}$ defined as follows:
\[
\|v\|_{\overline{V}_s}^2:=\|v\|_{(\mu,\phi)}^2+s|v(x)|^2_\phi,\quad \text{for any } v\in H^0(\X,\mathcal{L})\otimes_\Z\R.
\]

From \eqref{poisson}, we get
\begin{equation}\label{poisson1}
\sum_{v\in H^0(\X,\mathcal{L})} e^{-\pi\|v\|_{\overline{V}_s}^2} =\frac{1}{\mathrm{covol}(\overline{H^0(\X,\mathcal{L})}_s ) }
\sum_{v^\vee\in H^0(\X,\mathcal{L})^\vee} e^{-\pi \|v^\vee\|^2_{\overline{V}_s^\vee }}.
\end{equation}
Let $\{s_1,\ldots,s_N\}$ be a $\Z$-basis of $H^0(\X,\mathcal{L})$. We consider the following matrices.
\[
M:=(\left<s_i,s_j\right>_{(\mu,\phi)})_{1\leq i,j\leq N}\;\text{and}\; C:=(\left<s_i(x),s_j(x)\right>_{\phi})_{1\leq i,j\leq N}.
\]
Then, \eqref{poisson1} becomes
\[
\sum_{v\in H^0(\X,\mathcal{L})} e^{-\pi \left(\|v\|^2_{\phi }+ s\|v(x)\|_\phi^2\right)} =\frac{1}{\det(M+s C)^{\frac{1}{2}} }
\sum_{a\in \Z^N} e^{-\pi a^t(M+s C)^{-1}a },
\]
where $a^{t} $ denotes the transpose of the vector 
column $a\in \Z^N$.\\

By taking the derivative with respect to $s$, the preceding equation yields to
\begin{equation}\label{thetrho}
\sum_{v\in H^0(\X,\mathcal{L})} \|v(x)\|_\phi^2 e^{-\pi \|v\|^2_{\phi }}=\frac{\mathrm{Tr}(M^{-1} C)}{2\pi} \sum_{v\in H^0(\X,\mathcal{L})}  e^{-\pi \|v\|^2_{\phi }}-\frac{1}{\det(M)^{\frac{1}{2}}}  \sum_{a\in \Z^N}  \left(a^t M^{-1} C M^{-1} a\right) e^{-\pi  a^t  M^{-1} a }.
\end{equation}
Note that
\[
\frac{1}{\det(M)^{\frac{1}{2}}} \frac{\sum_{a\in \Z^N}  \left(a^t M^{-1} C M^{-1} a\right) e^{-\pi  a^t  M^{-1} a }}{\sum_{v\in H^0(\X,\LL)} e^{-\pi \|v\|_{(\mu,k\phi)}^2} }=
\frac{\sum_{a\in \Z^N}  \left(a^t M^{-1} C M^{-1} a\right) e^{-\pi  a^t  M^{-1} a }}{\sum_{a\in \Z^N}  e^{-\pi  a^t  M^{-1} a }  }.
\]
Then,
\[
\Theta(\mu, \phi)(x)
\leq \mathrm{Tr}(M^{-1} C).
\]

 Let $v_1,\ldots ,v_N$ be an orthonormal basis of $H^0(\X,\LL)\otimes_\Z \C$ with respect to the quadratic form defined by  $M$.  By basic arguments of linear algebra, there exists 
a $N\times N$ complex matrix $A=(a_{ij})_{1\leq i,j\leq N}$ such that
$M= A \overline{A}^t$. More precisely, we have
\[
s_i=\sum_{j=1}^N a_{ij} 
v_j,\quad \text{for}\; i=1,\ldots,N.\] 
It is easy to see that
 \[(\left<s_i(x),s_j(x)\right>_{\phi}))_{1\leq i,j\leq N}= A \bigl( \left<v_i(x),v_j(x) \right>_\phi\bigr)_{1\leq i,j\leq N} \overline A^t\quad \text{for any}\; x\in X.
\]
Hence
\[
\mathrm{Tr}(M^{-1} C)=\sum_{k=1}^N \|v_k(x)\|_\phi^2.
\]
But, we know that
\[
 \rho(\mu,\phi)(x)=\sum_{k=1}^N \|v_k(x)\|_\phi^2.
\]
We conclude that
\[
\Theta(\mu,\phi)(x)\leq   \rho(\mu,\phi)(x) \quad\text{for all}\; x\in X.
\]
\end{proof}

\begin{corollary}\label{upperTheta} We assume that $\LL_\Q$ is big. If  $\mu$ is  smooth, and $\phi$ is $\mathcal{C}^2$,
then
\begin{enumerate}[label=(\roman*)]

\item \[
\underset{k\rightarrow \infty}{\limsup}  \frac{\Theta(\mu,k\phi)}{k^n} \mu
\leq \mu_{\mathrm{eq}}(X,\phi).
\]

\item \[
\sup_{x\in X} \Theta(\mu,k\phi)=O(k^n),
\]
as $k\rightarrow \infty$.

\end{enumerate}

\end{corollary}

\begin{proof} This corollary is a direct consequence of Theorem \ref{thetarho} and 
\cite[Theorem 3.1]{BermanBoucksom}.
\end{proof}

\begin{theorem}\label{ample}

Let $\overline{\LL}=(\LL,\|\cdot\|_\phi)$ be an ample  $\mathcal{C}^\infty$-line bundle 
on $\X$.  We assume that 
$\mu$ is smooth. We have  
\[
\underset{k\rightarrow \infty}{\lim}  \frac{\Theta(\mu,k\phi)}{k^n} \mu
=\mu_{eq}(X,\phi).
\]
\end{theorem}

\begin{proof} We keep the same notations as in the proof of Theorem \ref{thetarho}.\\

There is a constant $\delta>0$, such that
\[
C\leq \delta  k^n  M,\quad \forall x\in X,\forall k\gg 1.
 \]
Indeed, let $z\in \C^n\setminus\{0\}$, we have
\[
\frac{\overline z^tC z}{\overline z^t M z}\leq  \rho(k\phi,\mu)\leq \delta k^n,
\]
where the second inequality is because \cite[Theorem 3.1]{BermanBoucksom}.\\

It follows that
\[
a^t M^{-1}C M^{-1}a\leq \delta k^n a^t M^{-1}a,\quad \forall a\in \Z^{N_k},\;\forall k\gg 1.
\]

Hence, 

\begin{equation}\label{mcmleq}
\frac{\sum_{a\in \Z^N}  \left(a^t M^{-1} C M^{-1} a\right) e^{-\pi  a^t  M^{-1} a }}{\sum_{a\in \Z^N}   e^{-\pi  a^t  M^{-1} a }  }\leq  \delta k^n 
\frac{\sum_{a\in \Z^N}  \left(a^t M^{-1}  a\right) e^{-\pi  (a^t  M^{-1} a) }}{\sum_{a\in \Z^N}   e^{-\pi  a^t  M^{-1} a }  }
\end{equation}


 By assumption,
there is $k_0\geq 1$ and $0<\alpha<1$, such that for any $k\geq k_0$, the $\Z$-module
$H^0(\X,k\LL)$ is generated by sections $e_1,\ldots,e_{N_k}$ such that 
$\|e_i\|_{\sup,k\phi}\leq \alpha^k$ for $i=1,2,\ldots,
N_k$. Let $\{e_1^\vee, \ldots, e_{N_k}^\vee\}$ be the $\Z$-basis of 
$H^0(\X,k\LL)^{\vee}:=
\mathrm{Hom}_\Z(H^0(\X,k\LL),\Z)$, dual to $\{e_1,\ldots,e_{N_k}\}$.\\

Let $k\geq k_0$. 
Let $v^\vee \in  H^0(\X,k\LL)^{\vee}$. We write $v^\vee=\sum_{i=1}^{N_k} a_i e_i^\vee$ where 
$a_1,\ldots,a_{N_k}\in \Z$.  
Set 
$a=(a_1,\ldots,a_{N_k})^t$. We have

\[
\begin{split}
(a^t M^{-1}a)^{\frac{1}{2}}=&\|v^\vee\|_{\overline{H^0(\X,k\LL)}^\vee_{(\mu, k\phi)}}\\
=&\underset{u\in H^0(\X,k\LL)\otimes_\Z\R }{\sup } \frac{|v^\vee(u)|}{\|u\|_{(\mu,k\phi)}}\\
& \geq \max_{i=1,\ldots,N_k} \frac{|a_i|}{\|e_i\|_{(\mu,k\phi)}} \\
&\geq \frac{1}{N_k}\left( \sum_{i=1}^{N_k} 
\frac{|a_i|^2}{\|e_i\|_{(\mu,k\phi)}^2}\right)^{\frac{1}{2}}\
\\&\geq 
\frac{1}{N_k \alpha^k} \left(\sum_{i=1}^{N_k} |a_i|^2\right)^{\frac{1}{2}},
\end{split}
\]
(where ${\overline{H^0(\X,k\LL)}^\vee_{(\mu, k\phi)}}$ denotes  the  hermitian vector bundle dual to $\overline{H^0(\X,k\LL)}_{(\mu,k\phi)}$).

Hence,
\begin{equation}\label{ineq1}
a^t M^{-1}a \geq \frac{1}{N_k^2\alpha^{2k}} a^t a.
\end{equation}

From this inequality, we get
 \[(1-\alpha^{2k})a^tM^{-1} a\geq \frac{a^t a}{N_k^2}(\frac{1}{\alpha^{2k}}-1).\]
That is
 \[a^t M^{-1}a
 \geq (\frac{1}{\alpha^{2k}}-1) \frac{a^t a}{ N_k^2}+\alpha^{2k} a^t M^{-1} a.\]

On the other hand, we have 
\[
\|v^\vee\|_{\overline{H^0(\X,k\LL)}^\vee_{(\mu,k\phi)}}\leq N_k \max_i |a_i|\|
e_i^\vee\|_{(\mu,k\phi)}\leq N_k \left(\sum_i |a_i|^2 \|e_i^\vee\|^2 
_{(\mu,k\phi)}\right)^{\frac{1}{2}}.
\]


We have 
\[
\begin{split}
\sum_{a\in \Z^{N_k}}  \left(a^t M^{-1}  a\right) e^{-\pi  a^t  M^{-1} a } &\leq  \sum_{a\in \Z^{N_k}} \left(a^t M^{-1}  a\right) e^{-\pi \left( \frac{(\alpha^{-2k}-1)}{N_k^2} a^ta+   \alpha^{2k}a^t   M^{-1} a \right)} \quad\text{by \eqref{ineq1}}\\
&\leq  e^{-\pi (\frac{\alpha^{-2k}-1}{N_k^2})} \sum_{a\in \Z^{N_k}\setminus\{0\}}  \left(a^t M^{-1}  a\right) e^{-\pi  \alpha^{2k} a^t  M^{-1} a }\\
&\leq  \frac{N_k}{2\pi \alpha^{2k}}e^{-\pi (\frac{\alpha^{-2k}-1}{N_k^2})}   \sum_{a\in \Z^{N_k}}   e^{-\pi \alpha^{2k} a^t  M^{-1} a }\quad \text{by \eqref{ineq2}}\\
&\leq   \frac{N_k}{2\pi \alpha^{2k}}e^{-\pi (\frac{\alpha^{-2k}-1}{N_k^2})}   \sum_{a\in \Z^{N_k}}  
e^{-\frac{\pi}{N_k^2} a^ta} \quad\text{by} \; \eqref{ineq1}\\
&= \frac{N_k}{2\pi \alpha^{2k}}e^{-\pi (\frac{\alpha^{-2k}-1}{N_k^2})} \left( \sum_{n\in \Z} e^{-\frac{\pi}{N_k^2}n^2} \right)^{N_k}\\
&= \frac{N_k}{2\pi \alpha^{2k}}e^{-\pi (\frac{\alpha^{-2k}-1}{N_k^2})} N_k^{N_k}
\left( \sum_{n\in \Z} e^{-\pi N_k^2 n^2} \right)^{N_k} \\
 &\leq \frac{N_k^{N_k+1}}{ 2\pi \alpha^{2k}}e^{-\pi (\frac{\alpha^{-2k}-1}{N_k^2})} 
\left(1+2e^\pi e^{-\pi N_k^2} \sum_{n=1}^\infty e^{-\pi n^2}\right)^{N_k}
\end{split}
\]
where we have used $\theta(t)=\frac{1}{\sqrt{t}}\theta(\frac{1}{t})$ where 
$\theta(t)=\sum_{n\in \Z} e^{-\pi n^2t}\quad (t>0).$

Note that 
    \[
    \lim_{k\rightarrow \infty} \left(1+2e^\pi e^{-\pi N_k^2} \sum_{n=1}^\infty e^{-\pi n^2}\right)^{N_k}=1.
    \]
Let $l\in \R$. We have
    \[
\lim_{k\rightarrow \infty}    \frac{\alpha^{2(k+1) } N_{k+1}^l\log N_{k+1}}{\alpha^{2k} N_k^l \log N_k}=\alpha^2,
    \]
     where    we have used that  $N_k=c k^n+O(k^{n-1})$.
     
      By letting, 
   $x_k=1-\frac{1}{\pi} \alpha^{2k} N_k^3\log N_k$ for any $k\in \N$, we deduce that
  \[
  \lim_{k\rightarrow \infty} x_k=1.
  \]     
     
  Note that
      \[
  N_k\log N_k-\frac{\pi}{N_k^2 \alpha^{2k}}=-\frac{\pi}{N_k^2\alpha^{2k}} x_k.
    \]

    Then, for any $l\in \R$,
    \[
    \begin{split}
   & \frac{\frac{N_{k+1}^{l+N_{k+1}}}{\alpha^{2(k+1)}} e^{-\pi (\frac{\alpha^{-2(k+1)}-1}{N_{k+1}^2})}  }{\frac{N_k^{l+N_k}}{\alpha^{2k}}e^{-\pi (\frac{\alpha^{-2k}-1}{N_k^2})} }\\
    =& 
  \alpha^{-2}\frac{N_{k+1}^l }{N_k^l}  
  \exp\left( N_{k+1}\log N_{k+1}-N_k \log N_k -\pi \frac{1}{\alpha^{2(k+1)}N_{k+1}^2 }  +\pi \frac{1}{\alpha^{2k}N_{k}^2 } +\frac{\pi}{N_{k+1}^2} -\frac{\pi}{N_{k}^2}  \right)  \\
  =& \alpha^{-2}\frac{N_{k+1}^l }{N_k^l}   \exp(\frac{\pi}{N_{k+1}^2} -\frac{\pi}{N_{k}^2})   \exp\left( N_{k+1}\log N_{k+1}-N_k \log N_k -\pi \frac{1}{\alpha^{2(k+1)}N_{k+1}^2 }  +\pi \frac{1}{\alpha^{2k}N_{k}^2 }  \right)  \\
    =& \alpha^{-2}\frac{N_{k+1}^l }{N_k^l}   \exp(\frac{\pi}{N_{k+1}^2} -\frac{\pi}{N_{k}^2})   \exp\left(  -\frac{\pi}{\alpha^{2(k+1)}N_{k+1}^2 } x_{k+1}+\frac{\pi}{\alpha^{2k}N_{k}^2 } x_{k}  \right)  \\
    =& \alpha^{-2}\frac{N_{k+1}^l }{N_k^l}   \exp(\frac{\pi}{N_{k+1}^2} -\frac{\pi}{N_{k}^2}) \exp\left(-\pi\frac{x_{k+1} }{N_{k+1}^2 \alpha^{2k+2}} \left(1-\alpha^2\frac{x_k N_{k+1}^2}{x_{k+1} N_k^2}\right) \right).
    \end{split}
    \]
    
That is
    
    \[
    \frac{\frac{N_{k+1}^{l+N_{k+1}}}{\alpha^{2(k+1)}} e^{-\pi (\frac{\alpha^{-2(k+1)}-1}{N_{k+1}^2})}  }{\frac{N_k^{l+N_k}}{\alpha^{2k}}e^{-\pi (\frac{\alpha^{-2k}-1}{N_k^2})} }=
    \alpha^{-2}\frac{N_{k+1}^l }{N_k^l}   \exp(\frac{\pi}{N_{k+1}^2} -\frac{\pi}{N_{k}^2}) \exp\left(-\pi\frac{x_{k+1} }{N_{k+1}^2 \alpha^{2k+2}} \left(1-\alpha^2\frac{x_k N_{k+1}^2}{x_{k+1} N_k^2}\right) \right)
    \]
    It follows that
        \[
    \frac{\frac{N_{k+1}^{l+N_{k+1}}}{\alpha^{2(k+1)}} e^{-\pi (\frac{\alpha^{-2(k+1)}-1}{N_{k+1}^2})}  }{\frac{N_k^{l+N_k}}{\alpha^{2k}}e^{-\pi (\frac{\alpha^{-2k}-1}{N_k^2})} }
     \sim  \alpha^{-2} \exp\left(-\frac{\pi}{N_{k+1}^2\alpha^{2k+2}}(1-\alpha^2)  \right),
    \]
    as $k\rightarrow \infty.$\\

Since
        $
    \lim_{k\rightarrow \infty} N_k^2 \alpha^{2k}=0$, and $0<\alpha<1$, we deduce that
    \[
    \lim_{k\rightarrow \infty}  \frac{\frac{N_{k+1}}{\alpha^{2(k+1)}}e^{-\pi (\frac{\alpha^{-2(k+1)}-1}{N_{k+1}^2})} N_{k+1}^{N_{k+1}} }{\frac{N_k}{\alpha^{2k}}e^{-\pi (\frac{\alpha^{-2k}-1}{N_k^2})} N_k^{N_k} }=0.
    \]

We conclude, using \eqref{mcmleq}, that 
\[
\lim_{k\rightarrow \infty} \frac{1}{k^n}\frac{\sum_{a\in \Z^N}  \left(a^t M^{-1} C M^{-1} a\right) e^{-\pi  a^t  M^{-1} a }}{\sum_{a\in \Z^N}   e^{-\pi  a^t  M^{-1} a }  }=0.
\]

     Then, the theorem follows from \eqref{thetrho}.
     \end{proof}

\begin{theorem}\label{Mor}
Let $\mu$ be a probability measure which has the Bernstein-Markov property with respect to the metric of $\LL$. We have 
\[
\widehat{\mathrm{vol}}(\overline{\LL})=\widehat{\mathrm{vol}}_{(\mu,k\phi)}(\overline{\LL}).
\]

\end{theorem}

\begin{proof}
The proof is essentially the same as in \cite[Lemma 2.1]{Moriwaki}.
\end{proof}

\begin{theorem}\label{main}
Let $(\LL, \|\cdot\|_\phi)$ and 
$(\LL, \|\cdot\|_\psi)$ be two $\mathcal{C}^\infty$ hermitian line bundles on $\X$. We assume $\psi\leq \phi$ and  $(\LL, \|\cdot\|_\psi)$ is generated by small sections. We have 
\[
\widehat{\mathrm{vol}}(\LL, \|\cdot\|_\phi)-\widehat{\deg}(\hat c_1( \LL, \|\cdot\|_{P_X\phi})^{n+1})
=
\widehat{\mathrm{vol}}(\LL, \|\cdot\|_\psi)-\widehat{\deg}(\hat c_1(\LL, \|\cdot\|_{P_X\psi })^{n+1}).
\]

\end{theorem}

\begin{proof}
Let $\eps>0$. 
Let $t\in [0,1]$, the hermitian line 
 $(\LL,\|\cdot\|_{t(\phi+\eps)+(1-t) \psi})$ 
 is generated by strictly small sections.
 
Combining  Proposition \ref{difference},  Corollary \ref{upperTheta}, Theorems \ref{ample} and 
\ref{Mor}, we get

\[
\widehat{\mathrm{vol}}(\LL, \|\cdot\|_\phi)-
\widehat{\mathrm{vol}}(\LL, \|\cdot\|_\psi)=\int_X (\phi-\psi)\int_0^1 \mu_{eq}(X,\phi_t)dt.
\] 
 It is known that
 \[
 \int_X (\phi-\psi)\int_0^1 \mu_{eq}(X,\phi_t)dt=
 \int_X (\phi-\psi) \sum_{i=0}^{n} (dd^c  P_X\phi)^i 
 (dd^c P_X\psi)^{n-1},
 \]
(see \cite[Proposition 4.1]{BermanBoucksom}).\\

Using \eqref{vardegree}, we conclude the proof
of the theorem.

\end{proof}

\section{A generalized Hodge index theorem on arithmetic toric varieties over $\Z$ }

Let $\X$ be an arithmetic toric variety over $\Z$. 
Let $\LL$ be an equivariant line bundle on $\X$.  It is known that 
$\LL$ possesses a canonical  and continuous metric which is  described uniquely in terms of the combinatorial structure of $\X$ (see \cite{Maillot, Zhang}). Let $\phi_\infty$ be  the weight of the canonical metric of $\LL$. We denote by 
$\|\cdot\|_{\phi_\infty}$ the canonical metric of
$\LL$, and we set
 $\overline{\LL}_\infty:=(\LL,\|\cdot\|_{\phi_\infty})$.

\begin{example}\label{canonical}
Let $n\in \N^\ast$. We endow the line bundle  $\mathcal{O}(1)$ on $\p^n$ with the metric
\[
 \|s(x)\|_{\phi_\infty}=\frac{|s(x)|}{\max(|x_0|,\ldots,|x_n|)},
\]
where  $s$ is  a local holomorphic section of $\mathcal{O}(1)$. Then $\vc_{\phi_\infty}$
is  admissible.
\end{example}

 We assume moreover  that $\LL$  is big. If $\LL$ is generated by its global sections, then the current $c_1(\overline{\LL}_\infty)$ is semi-positive, and it defines a measure $\mu_\LL$ on $X$:

\[
\mu_{\LL}
:=\frac{1}{\mathrm{vol}(L)} c_1(\overline{\LL}_\infty)^n.
\]

One can attach to $\LL$  a convex polytope 
denoted by $\Delta_\LL$ which describes 
the global sections of $k\LL$, more precisely 
\[
H^0(\X,k\LL)=\bigoplus_{m\in (k\Delta)\cap \Z^n} \Z \chi^m
\]
(see \cite{Fulton,Oda} for more details).

\begin{proposition}\label{producttoric}
Let $\X$ be an arithmetic  toric variety over $\Z$.  Let $\LL$ be an  equivariant big line
bundle generated by its global sections  on $\X$.   We consider the continuous hermitian line bundle  $\overline{\LL}_\infty=(\LL,\|\cdot\|_{\phi_\infty})$. Let $k\geq 1$. There exists an orthogonal basis $\{s_1,\ldots,s_{N_k}\}$ of elements of $  H^0(\X,k\LL)$ with respect to
the scalar product
$<\cdot,\cdot>_{(\mu_{\LL},k\phi_\infty)}$ such that
\[\int_X \frac{s_i(x) \overline{s_j(x)}}{\max(|s_1(x)|,\cdots,|s_N(x)|)^2} \mu_{\LL}=0\quad\forall 1\leq i\neq j\leq N_1.\]

\end{proposition}

\begin{proof}
We know that

\[
\|s\|_{\phi_\infty}(x)=\frac{|s(x)| }{ \max_{m\in \Delta_{\LL}\cap \Z^n}(|\chi^{m}(x)|)},\quad \forall \;s\in H^0(X,L),\quad\forall x\in X,
\]
 (see \cite[\S 3.3.3]{Maillot}). 

 Let  $m, m'
\in \Delta_{\LL}\cap \Z^n$. From \cite[Corollaire 6.3.5]{Maillot} and the fact that $|\chi^m(x)|=1$ on $S_{X}$   the compact torus  of $X$, we get
\[
\int_X \frac{\chi^{m} (x) \overline{\chi^{m'}(x)}}{\max(|\chi^{m_1}(x)|,\cdots,|\chi^{m_N}(x)|)^2} \mu_{\LL}
=\int_{S_{X}} \chi^{m}(x) \overline{\chi^{m'}(x)} \mu_{\LL}=  \delta_{m,m'},
\]

\end{proof}


\begin{theorem}\label{vol}
Let $\X$ be an arithmetic  toric variety over $\Z$.  Let $ \overline{\LL}_\infty=(\LL,\|\cdot\|_{\phi_\infty})$ be an  equivariant ample line
bundle on $\X$.  We have

\begin{enumerate}[label=(\roman*)]

\item

\[
\sup_{x\in X}\rho(\mu_{\LL},k \phi_\infty)(x)=
 \#((k\Delta_L)\cap \Z^n)=O(k^n)\quad\forall k\in \N.
\]

\item 
\[
\underset{k\rightarrow \infty}{\lim} \frac{h_\theta^0(
\overline{H^0(\X,k\mathcal{L})}_{(\mu_\LL, k\phi_\infty )})}{k^{n}/n!} = \mathrm{vol}(\LL) \log \sum_{n\in \Z} e^{-\pi n^2}, 
\]

\item
\[
\widehat{\mathrm{vol}}(\overline{\LL}_{\infty})=0.
\]
\end{enumerate}

\end{theorem}

\begin{proof}
\item

\begin{enumerate}[label=(\roman*)]

\item 
Recall that

\[
\rho(\mu_{\LL},k \phi_\infty)(x)=\sup_{s\in H^0(X, k L)} 
\frac{\|s(x)\|^2_{\phi_\infty}}{\|s\|^2_{(\mu_{\LL},\phi_\infty)}},
\]

Let $s\in H^0(X,kL)\setminus\{0\}$. We have  $s=\sum_{m\in (k\Delta)\cap\Z^n} a_m \chi^m$ where $a_m\in \C$ for any $m\in (k\Delta)\cap\Z^n$.  We have

\[
\begin{split}
\frac{\|s(x)\|^2_{k\phi_\infty}}{\|s\|^2_{(\mu_\LL,k\phi_\infty)}  }=&\frac{\sum_{m,m'\in (k\Delta_L)\cap \Z^n} a_m \overline {a_{m'}} \chi^m(x)\overline{\chi^{m'}(x)} e^{-2 k \phi_\infty}   }{(\sum_{m\in (k\Delta_L)\cap \Z^n} |a_m|^2 )  }\\
&\leq \sum_{m\in (k\Delta_L)\cap \Z^n} \|\chi^m(x)\|^2_{k\phi_\infty}.
\end{split}
\]
We conclude that
\[
\rho(\mu_\LL,k\phi_\infty)(x)=\sum_{m\in (k\Delta_L)\cap \Z^n} \|\chi^m(x)\|^2_{k\phi_\infty}.
\]

It follows that
\[
\sup_{x\in \X(\C)}\rho(\mu_{\LL},k \phi_\infty)(x)=
 \#((k\Delta_L)\cap \Z^n)=O(k^n).
\]

\item
From Proposition \ref{producttoric}, we see that

\[
\begin{split}
\sum_{s\in H^0(\X,k\LL)}e^{-\pi \|s\|_{(\mu_\LL,k\phi_\infty)}^2}=&\sum_{ 
(a_m)_{m\in (k\Delta\cap \Z^n)}\in \Z^{N_k}} e^{-\pi \sum_{m\in (k\Delta)\cap\Z^n} a_m^2}\\
=& \left( \sum_{n\in \Z} e^{-\pi n^2} \right)^{N_k}.
\end{split}
\]
Hence
\[
\underset{k\rightarrow \infty}{\lim} \frac{h_\theta^0(
\overline{H^0(\X,k\mathcal{L})}_{(\mu, k\phi_\infty )})}{k^{n}/n!} = \mathrm{vol}(\LL) \log \sum_{n\in \Z} e^{-\pi n^2}.
\]
\item 
By Theorem \ref{thm3.1} and $(ii)$, we conclude that
\[
\widehat{\mathrm{vol}}(\overline{\LL}_{\infty})=\lim_{k\rightarrow \infty} \frac{N_k}{k^{n+1}}\log(\sum_{n\in \Z} e^{-\pi n^2} )=0.
\]

\end{enumerate}

\end{proof}

\begin{theorem}\label{thm314}
Let $\X$ be an arithmetic toric variety over $\Z$.  Let $(\LL,\|\cdot\|_\phi)$ be an  equivariant ample hermitian line
bundle on $\X$. We have

\[
\widehat{\mathrm{vol}}(\LL, \|\cdot\|_\phi)=\widehat{\deg} (\hat c_1(\LL, \|\cdot\|_{P_X\phi})^{n+1}). \]
 If moreover, the metric of 
 $(\LL,\| \cdot \|_\phi)$ 
 is semipositive, then
\[
\widehat{\mathrm{vol}}(\LL, \|\cdot\|_\phi)=\widehat{\deg}(\hat c_1(\LL, \|\cdot\|_{\phi})^{n+1}).\]

\end{theorem}

\begin{proof}


There exists $0<\alpha\leq 1$ such that
\[
\alpha \|\cdot\|_\phi\leq \|\cdot\|_{\phi_\infty} 
\]
where $\|\cdot\|_{\phi_\infty}$ is the canonical metric
of $\LL$.\\

By Theorem \ref{main}, and 
 a continuity argument, we can  show that

\[
\widehat{\mathrm{vol}}(\LL, \alpha\|\cdot\|_\phi)-\widehat{\deg}(\hat c_1(\LL, \alpha \|\cdot\|_{P_X\psi })^{n+1}) =
\widehat{\mathrm{vol}}(\LL, \|\cdot\|_{\phi_\infty})-\widehat{\deg}(\hat c_1(\LL, \|\cdot\|_{P_X\phi_\infty})^{n+1}),
\]
and
\[
\widehat{\mathrm{vol}}(\LL, \alpha\|\cdot\|_\phi)-\widehat{\deg}(\hat c_1(\LL, \alpha \|\cdot\|_{P_X\phi })^{n+1})=
\widehat{\mathrm{vol}}(\LL, \|\cdot\|)-\widehat{\deg}(\hat c_1(\LL, \|\cdot\|_{P_X\phi })^{n+1}).
\]

By Theorem \ref{main}, (ii) of Theorem \ref{vol} and \cite[Proposition 7.1.1]{Maillot}, we deduce 
\[
\widehat{\mathrm{vol}}(\LL, \|\cdot\|)=\widehat{\deg}(\hat c_1(\LL, \|\cdot\|_{P_X\phi })^{n+1}).
\]

\end{proof}

\begin{theorem}\label{Hodge}[A generalized Hodge index theorem]
 Let $\X$ be an arithmetic toric variety over  $\Z$ of relative dimension $n$. Let  $\mathcal{L}$ be an equivariant ample line bundle  on $\X$.
We   assume that  the metric of  $\overline{\mathcal{L}}$ is  smooth  with semi-positive  first Chern form $c_1(\overline{\mathcal{L}})$  on $\X$, then
\[
\widehat{\mathrm{vol}}(\overline{\mathcal{L}})\geq \widehat{\deg}(\widehat c_1(\overline{\mathcal{L}})^{n+1}).
\]
\end{theorem}

\begin{proof}
There exist $k_0\in \N$, and $s_1,\ldots,s_{N_{k_0}}$ 
elements of $H^0(\X,k_0\LL)$ such that the $\Z$-algebra 
$\bigoplus_{k\geq 0} H^0(\X,kk_0\LL)$ is generated by $s_1,\ldots,s_{N_{k_0}}$.

Since, for any $k=1,2,\ldots,$
\[
\widehat{\mathrm{vol}}(k\overline{\mathcal{L}})=k^{n+1} 
\widehat{\mathrm{vol}}(\overline{\mathcal{L}})\quad\text{and}\quad \widehat{\deg}(\widehat c_1(k\overline{\mathcal{L}})^{n})=k^{n+1} \widehat{\deg}(\widehat c_1(\overline{\mathcal{L}})^{n}),
\]
we may assume that $k_0=1$. Let $\alpha$ be a positive real number such that
\[
0<\alpha<1\quad\text{and}\quad \alpha \|s_i\|_{\sup,\phi}<1\quad\text{for}\quad i=1,2,\ldots,N_1.\]
It follows   that the hermitian line bundle 
 $(\LL,\alpha \|\cdot\|_\phi)$ is ample. From Theorem \ref{thm314}, we get
 \[
\widehat{\mathrm{vol}}(\LL,  \alpha\|\cdot\|_\phi)=\widehat{\deg}( \hat c_1(\LL, \alpha\|\cdot\|_{\phi})^{n+1} ).
\]
 
The combination of  Proposition \ref{difference} and  Theorem \ref{thetarho} gives the following inequality
\[
\widehat{\mathrm{vol}}(\LL,  \|\cdot\|_\phi) -\widehat{\deg}(\hat c_1(\LL, \|\cdot\|_{\phi})^{n+1})\geq \widehat{\mathrm{vol}}(\LL,  \alpha\|\cdot\|_\phi)-\widehat{\deg}(\hat c_1(\LL, \alpha\|\cdot\|_{\phi})^{n+1}).\]
Therefore
\[
\widehat{\mathrm{vol}}(\LL,  \|\cdot\|_\phi)\geq 
\widehat{\deg}(\hat c_1(\LL, \|\cdot\|_{\phi})^{n+1}).
\]

\end{proof}

\bibliographystyle{plain} 

\bibliography{biblio}

\begin{thebibliography}{10}

\bibitem{Abbes}
A.~Abbes and T.~Bouche.
\newblock Th\'eor\`eme de {H}ilbert-{S}amuel ``arithm\'etique''.
\newblock {\em Ann. Inst. Fourier (Grenoble)}, 45(2):375--401, 1995.

\bibitem{Bedford}
Eric Bedford and B.~A. Taylor.
\newblock A new capacity for plurisubharmonic functions.
\newblock {\em Acta Math.}, 149(1-2):1--40, 1982.

\bibitem{BermanBoucksom}
Robert Berman and S{\'e}bastien Boucksom.
\newblock Growth of balls of holomorphic sections and energy at equilibrium.
\newblock {\em Invent. Math.}, 181(2):337--394, 2010.

\bibitem{BoGS}
J.-B. Bost, H.~Gillet, and C.~Soul{\'e}.
\newblock Heights of projective varieties and positive {G}reen forms.
\newblock {\em J. Amer. Math. Soc.}, 7(4):903--1027, 1994.

\bibitem{BostTheta}
Jean-Beno\^{\i}t Bost.
\newblock {\em Theta invariants of {E}uclidean lattices and
  infinite-dimensional {H}ermitian vector bundles over arithmetic curves},
  volume 334 of {\em Progress in Mathematics}.
\newblock Birkh\"{a}user/Springer, Cham, [2020] \copyright 2020.

\bibitem{Burgos3}
Jos{\'e}~Ignacio Burgos~Gil, Atsushi Moriwaki, Patrice Philippon, and
  Mart{\'{\i}}n Sombra.
\newblock Arithmetic positivity on toric varieties.
\newblock {\em J. Algebraic Geom.}, 25(2):201--272, 2016.

\bibitem{Chen1}
Huayi Chen.
\newblock Arithmetic {F}ujita approximation.
\newblock {\em Ann. Sci. \'Ec. Norm. Sup\'er. (4)}, 43(4):555--578, 2010.

\bibitem{MoriwakiAdelic}
Huayi Chen and Atushi Moriwaki.
\newblock {\em Arakelov geometry over adelic curves}, volume 2258 of {\em
  Lecture Notes in Mathematics}.
\newblock Springer, Singapore, [2020] \copyright 2020.

\bibitem{Fulton}
William Fulton.
\newblock {\em Introduction to toric varieties}, volume 131 of {\em Annals of
  Mathematics Studies}.
\newblock Princeton University Press, Princeton, NJ, 1993.
\newblock The William H. Roever Lectures in Geometry.

\bibitem{Amplitude}
Henri Gillet and Christophe Soul{\'e}.
\newblock Amplitude arithm\'etique.
\newblock {\em C. R. Acad. Sci. Paris S\'er. I Math.}, 307(17):887--890, 1988.

\bibitem{AIT}
Henri Gillet and Christophe Soul{\'e}.
\newblock Arithmetic intersection theory.
\newblock {\em Inst. Hautes \'Etudes Sci. Publ. Math.}, 72:93--174 (1991),
  1990.

\bibitem{Character}
Henri Gillet and Christophe Soul{\'e}.
\newblock Characteristic classes for algebraic vector bundles with {H}ermitian
  metric. {I}.
\newblock {\em Ann. of Math. (2)}, 131(1):163--203, 1990.

\bibitem{ARR}
Henri Gillet and Christophe Soul{\'e}.
\newblock An arithmetic {R}iemann-{R}och theorem.
\newblock {\em Invent. Math.}, 110(3):473--543, 1992.

\bibitem{MounirIJNT}
Mounir Hajli.
\newblock On an arithmetic inequality on {$\Bbb{P}_{\Bbb{Q}}^1$}.
\newblock {\em Int. J. Number Theory}, 11(4):1227--1231, 2015.

\bibitem{Maillot}
Vincent Maillot.
\newblock {G}\'eom\'etrie d'{A}rakelov des vari\'et\'es toriques et fibr\'es en
  droites int\'egrables.
\newblock {\em M\'em. Soc. Math. Fr. (N.S.)}, 80:vi+129, 2000.

\bibitem{Moriwaki1}
Atsushi Moriwaki.
\newblock Arithmetic height functions over finitely generated fields.
\newblock {\em Invent. Math.}, 140(1):101--142, 2000.

\bibitem{Moriwaki2}
Atsushi Moriwaki.
\newblock Continuity of volumes on arithmetic varieties.
\newblock {\em J. Algebraic Geom.}, 18(3):407--457, 2009.

\bibitem{MoriwakiIMRN}
Atsushi Moriwaki.
\newblock Continuous extension of arithmetic volumes.
\newblock {\em Int. Math. Res. Not. IMRN}, (19):3598--3638, 2009.

\bibitem{Moriwaki}
Atsushi Moriwaki.
\newblock Big arithmetic divisors on the projective spaces over {$\Bbb Z$}.
\newblock {\em Kyoto J. Math.}, 51(3):503--534, 2011.

\bibitem{MoriwakiAMS}
Atsushi Moriwaki.
\newblock Adelic divisors on arithmetic varieties.
\newblock {\em Mem. Amer. Math. Soc.}, 242(1144):v+122, 2016.

\bibitem{Oda}
Tadao Oda.
\newblock Convex bodies and algebraic geometry---toric varieties and
  applications. {I}.
\newblock In {\em Algebraic {G}eometry {S}eminar ({S}ingapore, 1987)}, pages
  89--94. World Sci. Publishing, Singapore, 1988.

\bibitem{ZhangPositive}
Shouwu Zhang.
\newblock Positive line bundles on arithmetic varieties.
\newblock {\em J. Amer. Math. Soc.}, 8(1):187--221, 1995.

\bibitem{Zhang}
Shouwu Zhang.
\newblock Small points and adelic metrics.
\newblock {\em J. Algebraic Geom.}, 4(2):281--300, 1995.

\end{thebibliography}

\end{document}